\documentclass[a4paper,11pt]{amsart}
\usepackage{amsthm,amsmath,amsfonts,amssymb,graphicx}

\usepackage{color}

\newlength{\hchng}
\newlength{\vchng}
\newtheorem{theorem}{Theorem}[section]

\newtheorem{corollary}[theorem]{Corollary}
\newtheorem{lemma}[theorem]{Lemma}
\newtheorem{definition}[theorem]{Definition}

\newtheorem{preremark}[theorem]{Remark}
\newtheorem{claim}[theorem]{Claim}

\newenvironment{remark}{\begin{preremark}\rm}{\medskip \end{preremark}}
\numberwithin{equation}{section}

\newcommand{\R}{\mathbb R}

\DeclareMathOperator{\reg}{reg}
\DeclareMathOperator{\divg}{div}

\begin{document}

\title{Streamlines concentration and application to the incompressible Navier-Stokes equations}
\author{Eric Foxall, Slim Ibrahim  and Tsuyoshi Yoneda}
\email{E. Foxall: e.t.foxall@gmail.com}
\email{S. Ibrahim: ibrahim@math.uvic.ca} \urladdr{
http://www.math.uvic.ca/~ibrahim/}
\email{T. Yoneda: yoneda@math.sci.hokudai.ac.jp}

\thanks{S. I. is partially supported by NSERC\# 371637-2009 grant and a start up fund from University
of Victoria}
\thanks{T. Y. is partially supported by PIMS Post-doc fellowship at the University
of Victoria, and partially supported by NSERC\# 371637-2009}
\maketitle
\begin{center}
Department of Mathematics and Statistics,
University of Victoria \\
PO Box 3060 STN CSC,
Victoria, BC, Canada, V8W 3R4 \\
\end{center}
\begin{center}
Department of Mathematics and Statistics,
University of Victoria \\
PO Box 3060 STN CSC,
Victoria, BC, Canada, V8W 3R4 \\
\end{center}
\begin{center}
Department of Mathematics, Hokkaido University\\ 
Sapporo 060-0810, Japan
\end{center}
\bibliographystyle{plain}



    \begin{abstract}
For a smooth domain $D$ containing the origin, we consider a vector field $u \in C^1(D\setminus\{0\},\mathbb{R}^3)$ with $\divg u \equiv 0$ and exclude certain types of possible isolated singularities at the origin, based on the geometry of streamlines that go near that possible singular point.
\end{abstract}

\section{Introduction}
In this paper  we consider  divergence-free smooth vector fields $u \in C^1(D\setminus\{0\},\mathbb{R}^3)$ defined on a domain $D$ of $\R^3$ containing the origin which may have a singular point at the origin.
We give a definition based on streamline concentration towards the eventual singularity, and we show that if there is sufficient streamline concentration, then the vector field cannot be an $L^2$ function\footnote{we define this situation precisely in the next section}. Therefore, this result rules out a certain geometric situation (streamline concentration) at a possible singular time for incompressible fluid equations such as the 3D Navier-Stokes equations. Before going any further, let us briefly recall a few results about the 3D Navier-Stokes equations on $\mathbb{R}^3$. The equations ruling the flow of an incompressible viscous fluid on $\R^3$ are
\begin{equation}\label{NS}
\begin{cases}
\partial_{t} v -\triangle v + \mbox{div} (v\otimes v) + \nabla p  = 0,  \\
\mbox{div} (v)  = 0,\quad v|_{t=0}=v_0

\end{cases}
\end{equation}
in which

$v$ is a vector-valued function representing the velocity of
the fluid, and $p$ is the pressure.
The  initial  value problem of the above
equation is endowed with the condition that $v(0, \cdot ) = v_{0}
\in L^2(\mathbb{R}^3)$.

A finite energy {\it weak solution} to the Navier-Stokes equations \eqref{NS} over a time interval $(0,T)$ is a pair $(v,p)$  satisfying

\begin{enumerate}
\item equation \eqref{NS} in the distributional sense,
\item $(v,p)\in L^\infty([0,T],L^2)\cap L^2([0,T],\dot H^1)\times L^{\frac53}_{loc}((0,T)\times\R^3)$
\item the energy inequality, for $0<t<T$
\begin{eqnarray}
\label{EnergyIne}
\|v(t,\cdot)\|_{L^2}^2+2\int_0^t\|\nabla v(t',\cdot)\|_{L^2}^2\;dt'\leq\|v(0,\cdot)\|_{L^2}^2.
\end{eqnarray}
\end{enumerate}
For a divergence free initial data $v_0\in (L^2(\R^3))^3$, the existence of global in time and finite energy  {\it weak solutions} to the Navier-Stokes equations is due to the pioneer works of J. Leray \cite{Leray} in the case $D=\R^3$ and E. Hopf \cite{Hopf} in the case of the torus. Moreover, neither the uniqueness nor the global regularity are known. These questions are the outstanding problems of regularity for solutions to the Navier-Stokes equations. Recall that the space-time singular set $S(u)$ of $u$ is defined as follows.

\begin{definition}
A point $(x_0,t_0)\notin S(u)$ if there exists a parabolic cylinder $Q_{(x_0,t_0)}(r):=\{|x-x_0|<r\}\times(t_0-r^2,t_0)$ about $(x_0,t_0)$ such that the solution $u\in L^\infty(Q_{(x_0,t_0)}(r))$.
\end{definition}
Modern regularity theory for solutions to equation \eqref{NS} began with the works of Prodi \cite{Prodi}, Serrin \cite{Serrin}, Ladyzhenskaya \cite{La}  implying that if
$$
u \in L^p_t(L^q_x)(Q_{(x_0,t_0)}(r)), \quad\mbox{for}\quad \frac3q+\frac2p<1,$$
then $\partial^k_xu\in\mathcal C^\alpha((Q_{(x_0,t_0)}(r/2)))$ for some $0<\alpha<1$ and therefore $u$ is regular. Later on, M. Struwe \cite{Struwe} extended this to the case (of scaling invariant pair)
i.e. $\frac3q+\frac2p=1$, and recently this was extended  to the limit case $u \in L^{\infty}_t(L^{3}_x)$ by L. Escauriaza, G. Seregin, and V. Sverak (see their famous work \cite{Esca}).
After the appearance of the Prodi-Serrin-Ladyzhenskaya criterion, many different regularity cirteria and Liouville type theorem of solutions to \eqref{NS} were established (see  \cite{Beale}, \cite{Veiga},  \cite{CSTY} and \cite{KNSS}).

We would like to mention a  regularity criterion in \cite{Vasseur2} by A. Vasseur (see also  \cite{Chan}).
He gave a regularity criterion for solutions $u$ to \eqref{NS} in terms of the
integral condition $\text{div}(\frac{u}{|u|}) \in L^{p}(0,\infty ;L^{q}(\mathbb{R}^{3}))$ with $\frac{2}{p}+ \frac{3}{q} \leqslant \frac{1}{2}$
imposed on the scalar quantity $F = \text{div} (\frac{u}{|u|})$. Note that the case $(p,q)=(6,\infty)$ is included.

Concerning the analysis of the singular set $S(u)$, we recall the following facts: First, by definition, the set $S(u)$ is closed, and thanks to the result of C. Foias and R. Temam \cite{Foias}, the $\frac12$-dimensional Hausdorff measure of the set of singular times $\tau(u):=\mbox{proj}_tS(u)$\footnote{the map $(x,t)\mapsto t$} is zero. Next, V. Scheffer \cite{Scheffer} and then L. Caffarelli, R. Kohn and L. Nirenberg \cite{CKN} showed the best result concerning {\it partial regularity} of {\it suitable weak} solutions\footnote{roughly, these are weak solutions satisfying the local energy inequality instead of the global one \eqref{EnergyIne}.} of the Navier-Stokes equations stating that the parabolic one-dimensional Hausdorff measure of $S(u)$ is zero. Finally, a consequence of the latter result  is a bound on the spatial singular set for each time slice  $S_T := S(u) \cap \{t = T \}$  which has at most one-dimensional Hausdorff measure.

In this paper, we focus on the vector field at a possible singular time $T\in\tau(u)$, and examine the geometry of its streamlines.  Recall that in \cite{CY}, C-H. Chan and the third author  proposed a possible scenario for an isolated space singularity at a possible blow-up time by using the energy inequality and regularity criterions especially \cite{Esca} and \cite{Vasseur2}. They constructed a divergence free velocity field $u$ within a {\it streamtube} segment
with increasing twisting (i.e., increasing swirl).

The construction of such a vector field $u$
demonstrates the way in which \emph{excessive} increase of twisting of streamlines
 can result in the \emph{blow up} of
the quantities $\|u\|_{L^{\alpha}(\mathbb{R}^{3})}$ (for some $2 < \alpha < 3$) and $\|\text{div} (\frac{u}{|u|})\|_{L^{6}(\mathbb{R}^{3})}$ while at the same time preserving
the finite energy property $u \in L^{2}(\mathbb{R}^{3})$ of the fluid.
Note that the increasing swirl streamtube is not included in the sufficient concentration streamlines  case. The device of streamtube has already proposed as the vortex-tube (see\cite{CM}).\\
In this work, we show that if ``enough" streamlines of a smooth and divergence free vector field concentrate towards a possible isolated singular point,\footnote{note that such singular set has a zero one-dimensional Hausdorff measure.} then the vector field cannot be an $L^2$ function. The main idea is to costruct an appropriate ``streamline flux tube" and apply Stokes' Theorem.

\section{A classification of divergence vector fields}

\begin{definition} (Streamline)
Let $D$ be a smooth domain containing the origin and $u:D\setminus\{0\}\to\mathbb{R}^3$ be a smooth vector field.
For a starting point $\eta\in D$, we define a  streamline $\gamma_\eta(s):[0,\infty)\to \mathbb{R}^3$ as the curve solving
\begin{equation}\label{ODE}
\partial_s\gamma_\eta(s)=u(\gamma_\eta(s))\quad\text{for}\quad s>0\quad\text{with}\quad \gamma_\eta(0)=\eta.
\end{equation}
\end{definition}
One may assume that streamlines are global, because otherwise, they go towards the possible singular point at  the origin.

The following definition is the key to classify the divergence-free vector field with a possible isolated singularity at the origin. Let $B_\alpha$ be the open ball with radius $\alpha$ centered at the origin.

\begin{definition}
  For  $\alpha>r$  let
\begin{equation*}
A_r^\alpha = \{\eta \in \partial B_{\alpha}: \gamma_\eta(s)\in
B_{r} \textrm{ for some }s>0,\;\mbox{and}\;
\gamma_\eta(s')\in B_\alpha\ {\text{for}\ 0<s'<s}\}.
\end{equation*}
\end{definition}
The above definition excludes the streamlines entering the ball $B_\alpha$ infinitely many times before entering $B_r$. If it happens and a streamline enters $B_\alpha$ finitely many times before getting into $B_r$, then one can re-parametrize the time so that its last entrance occurs at time $s=0$.

\begin{remark}
For streamlines from $A^\alpha_r$ we have the following properties
\begin{itemize}
\item
$|A^\alpha_r|$ is monotone decreasing with respect to $\alpha$ and increasing with respect to $r$. Indeed,
\begin{equation*}
 |A^\alpha_r|\geq |A^\alpha_{r'}|\quad\text{for}\quad r>r',\quad
 |A^\alpha_r|\geq |A^{\alpha'}_{r}|\quad\text{for}\quad \alpha<\alpha'.
 \end{equation*}
\item
Without loss of generality, we can assume that streamlines from $A^\alpha_r$ are globally defined.
\item
From definition of $A^\alpha_r$ we cannot have stagnation points of the fluid  (i.e. $u(\gamma_\eta(s))=0$ for some $s>0$).
\end{itemize}
\end{remark}

\begin{definition} (Stream-surface \& flux-tube)
Let $D \subset \mathbb{R}^3$ be a surface and $s$ be such that $\gamma_{\eta}(s)$ is defined for each $\eta \in D$.
\begin{itemize}
\item A {\it stream-surface} $S^D(s)$ is defined as $S^D(s) = \bigcup_{\eta \in D} \gamma_{\eta}(s)$.
\item A {\it flux-tube} $T^D(s)$ is given by $T^D(s) = \bigcup_{0 \leq s' \leq s}S^D(s')$.
\item The {\it mantle} of the  {\it flux-tube} $T^D(s)$ is $\partial T^D(s)$.
\end{itemize}
\end{definition}

For $|x| \neq 0$ denote by $\hat{n}(x) = x/|x|$.  Smoothness and membership in $C^1$ are used interchangeably. The main result reads as follows.

\begin{theorem}
\label{MainThm} 
If for some $\alpha>0$ and for some $C>0$ independent of $r$, $|\int_{A_r^{\alpha}} u\cdot\hat{n}d\sigma| \geq Cr^{1/2}$ as $r \rightarrow 0$, then $u \notin L^2(\mathbb{R}^3)$.
\end{theorem}

\begin{figure}
\centering
\includegraphics[height=80mm, width=70mm]{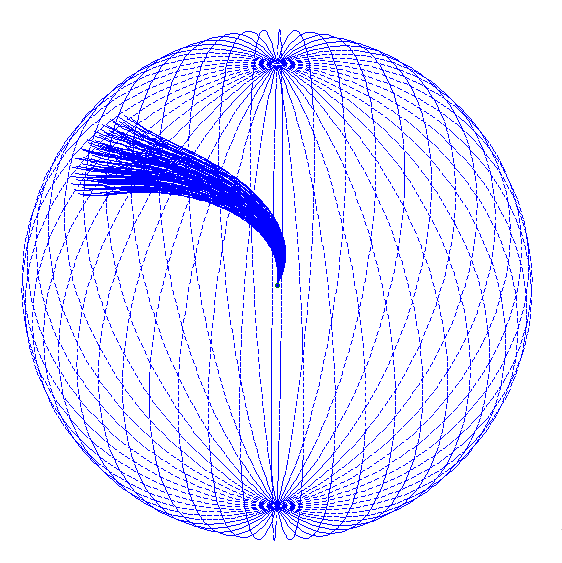}
\caption{The set $A$ of Corollary \ref{shrink}, with streamlines going to the origin}
\label{st_tube}
\end{figure}

The following special case is worth noting.  See Figure \ref{st_tube}.
\begin{corollary}\label{shrink}Suppose for some $\alpha>0$ and for $A \subset \partial B_{\alpha}$ that $\int_A u\cdot \hat{n}d\sigma \neq 0$ and $A_r^{\alpha} \supset A$ for $0<r<\alpha$.  Then $u \notin L^2(\mathbb{R}^3)$.
\end{corollary}
\begin{proof}
It follows from the definition of $A_r^{\alpha}$ that $u\cdot\hat{n}$ has constant (negative) sign on $A_r^{\alpha}$.  Let $C = |\int_A u\cdot \hat{n}d\sigma| >0$, then for $0<|r| < \min \{1,\alpha\}$, $|\int_{A_r^{\alpha}}u\cdot\hat{n}d\sigma| \geq |\int_A u\cdot\hat{n}d\sigma| \geq Cr^{1/2}$.
\end{proof}

The proof of Theorem \ref{MainThm} proceeds in a few steps.  First of all suppose that
\begin{equation*}
\int_{\partial B_r} |u\cdot \hat{n}|d\sigma \geq \left|\int_{A_r^{\alpha}} u\cdot \hat{n}d\sigma \right|
\end{equation*}
for each $r$ (this is proved in a moment).  Then, Jensen's inequality gives
\begin{equation}\label{Jensen1}
\frac{1}{|\partial B_r|} \int_{\partial B_r} |u|^2d\sigma \geq \left ( \frac{1}{|\partial B_r|}\int_{\partial B_r} |u|d\sigma \right ) ^2
\end{equation}
or
\begin{equation}\label{Jensen2}
\int_{\partial B_r}|u|^2 d\sigma \geq \left (\frac{1}{|\partial B_r|}\int_{\partial B_r} |u| d\sigma \right) ^2
\end{equation}
and by assumption
\begin{equation*}
\left (\frac{1}{|\partial B_r|}\int_{\partial B_r} |u| d\sigma \right) ^2 \geq \frac{1}{4\pi r^2}\left| \int_{A_r^{\alpha}} u\cdot \hat{n}d\sigma \right | ^2 \geq \frac{1}{4 \pi r^2} Cr = \frac{C}{4 \pi r}
\end{equation*}
from which it follows that
\begin{equation*}
\|u \|_{L^2} \geq \left( \int_0^{\epsilon} \int_{\partial B_r}|u|^2 d\sigma dr \right ) ^{1/2} \geq \left ( \int_0^{\epsilon} \frac {C}{4 \pi r} \right) ^{1/2} = \infty
\end{equation*}
where $\epsilon>0$ is such that $|\int_{A_r^{\alpha}} u\cdot \hat{n} d\sigma| \geq C r^{1/2}$ for $0<r\leq \epsilon$.\\

Now, to prove that $\int_{\partial B_r} |u\cdot \hat{n}|d\sigma \geq \left | \int_{A_r^{\alpha}} u\cdot \hat{n}d\sigma \right |$ observe first of all that $\int_{A_r^{\alpha}} u\cdot \hat{n}d\sigma = \int_{\reg A_r^{\alpha}} u\cdot \hat{n}d\sigma$ where $\reg A_r^{\alpha} = \{\eta \in A_r^{\alpha} : (u\cdot \hat{n})(\eta)\neq 0\}$.  Since $\alpha$ is fixed, let $A_r$ denote $\reg A_r^{\alpha}$.  From the definition of $A_r^{\alpha}$ it follows that $(u\cdot\hat{n})(\eta)<0$ for $\eta \in A_r$.\\

\begin{lemma}
\label{Lemma}Let $D\subset \partial B_{\alpha}$ have piecewise smooth boundary and $(u\cdot\hat{n})(\eta)<0$ for $\eta \in D$.  Suppose that $S^D(s) \subset B_r$ for some $s>0$ and that $S^D(s') \subset B_{\alpha}$ for $0<s'\leq s$.  Then
\begin{equation*}
\int_D u\cdot \hat{n}d\sigma = \int_{D^*} u\cdot \hat{n}d\sigma
\end{equation*}
where $D^* \equiv T^D(s) \cap \partial B_r$.  Also, if $D_1$ and $D_2$ are two such sets with $D_1 \cap D_2 = \emptyset$, then $D_1^* \cap D_2^* = \emptyset$.
\end{lemma}
\begin{proof}
The function $\gamma_{\eta}:D \times [0,s] \rightarrow T^D(s)$ is onto and it follows from the theory of ordinary differential equations and from $u \in C^1$ that $\gamma_{\eta} \in C^1$.  Also, $\gamma_{\eta}$ is injective, which follows from uniqueness of solutions and from the fact that for each $\eta \in D$, $\gamma_{\eta}(s) \notin D$ for $s>0$.  From these properties it can be shown that $\partial T^D(s) = D \cup S^D(s) \cup T^{\partial D}(s)$.  Piecewise smoothness of $\partial T^D(s)$ then follows from the piecewise smoothness of $\partial D$ and smoothness of solutions to the vector field.  Let $T = \{x \in T^D(s):r<|x|<\alpha\}$ and let $V = \{x \in T^{\partial D}(s): r<|x|<\alpha\}$, and let $D^*$ be as defined above.  Note that $T$ has piecewise smooth boundary since it is the intersection of two sets with piecewise smooth boundary.  Write $\partial T = D \cup D^* \cup V$.  If $x \in V$ then a part of the streamline through $x$ lies in $V$, therefore $u(x)$ is in the tangent space of $V$ at $x$.  Then, applying the divergence theorem and using $\divg u \equiv 0$ gives the stated result.  Observe that the implication $D_1\cap D_2 = \emptyset \Rightarrow D_1^*\cap D_2^* = \emptyset$ follows from the uniqueness of solutions in the same way as above.
\end{proof}

\begin{claim}$A_r$ is open.  Moreover, for each $\eta \in A_r$ there is a $\delta>0$ such that $D\equiv \{\xi \in \partial B_{\alpha}: |\xi-\eta|<\delta\}$ satisfies the assumptions of the above lemma.
\end{claim}
\begin{proof}
Let $\eta \in A_r$ and $s$ be as in the definition of $A_r^\alpha$. Then $(u\cdot \hat{n})(\eta) <0$.  By continuity there exists $\delta>0$ so that $E \equiv \{\xi \in \partial B_{\alpha}:|\xi-\eta|\leq \delta\}$ has $(u\cdot\hat{n})(\lambda)<0$ for $\xi \in E$.  $E$ is compact, and by a property of compact sets, there exists $\alpha>0$ so that $\mbox{dist}(\xi, E)<\alpha$ implies $(u\cdot\hat{n})(\xi)<0$.  Let $t = \inf\{s'>0: |\gamma_{\eta}(s') - \eta|>\alpha/2\}$ and let $\beta(s)= \inf\{|\gamma_{\eta}(s') - \partial B_{\alpha}|: t \leq s' \leq s\}$.  Observe that $\beta>0$ since the sets $\{\gamma_{\eta}(s'): t\leq s' \leq s\}$ and $\partial B_{\alpha}$ are compact and disjoint.  Let $\beta'>0$ be such that $|\xi-\gamma_{\eta}(s)|<\beta'$ implies $\xi \in B_r$.  Let $\alpha' = \min\{\alpha/2,\beta,\beta'\}$.  By continuous dependence on initial data, there is a $\delta'>0$, $\delta'\leq \delta$ so that $|\xi-\eta|<\delta'$ implies $|\gamma_{\xi}(s')-\gamma_{\eta}(s')|<\alpha'$ for $0\leq s' \leq s$.  For these $\xi$, $|\gamma_{\xi}(s') - E|<\alpha$ for $0\leq s' \leq t$ and so $(u\cdot \hat{n})(\gamma_{\xi}(s'))<0$ for $ 0\leq s' \leq t$, from which it follows that $\gamma_{\xi}(s') \in B_{\alpha}$ for $0< s' \leq t$.  Then, $|\gamma_{\xi}(s)-\gamma_{\eta}(s)|<\beta'$ implies $\gamma_{\xi}(s) \in B_r$, and $|\gamma_{\xi}(s') - \gamma_{\eta}(s')|<\beta$ implies $\gamma_{\xi}(s') \in B_{\alpha}$, for $t\leq s' \leq s$.  Therefore $\delta'$ gives $D$ that satisfies the claim.
\end{proof}

{\it End of the proof of Theorem \ref{MainThm}}. Since $A_r$ is open it is Lebesgue measurable.  It follows that for each $\epsilon>0$, by a theorem for measurable sets there exists $K$ closed, $K\subset A_r$ such that $m(A_r \setminus K)<\epsilon$, where $m$ denotes Lebesgue measure.  For each $\eta \in A_r$ let $D_{\eta}$ be as in the above claim, then $\{D_{\eta}\}_{\eta \in K}$ is an open cover of $K$.  Since $K$ is a closed and bounded subset of $\mathbb{R}^3$, it is compact and therefore from the above cover one can take a finite subcover $\{D_{\eta_i}\}_{1\leq i \leq k}$.  Let $E_1 = D_{\eta_1}$ and for $2 \leq i \leq k$ let $E_i = D_{\eta_i} \setminus E_{i-1}$; then the $E_i$ are pairwise disjoint and have piecewise smooth boundary, and $\bigcup_{i=1}^k E_i$ covers $K$.  For each $i$ let $E_i^* = T^{E_i}(s) \cap \partial B_r$.  Then
\begin{equation*}
\int_{\bigcup_{i=1}^k E_i} u\cdot \hat{n}d\sigma = \int_{\bigcup_{i=1}^k E_i^*} u\cdot \hat{n}d\sigma
\end{equation*}
using $\int_{E_i}u\cdot\hat{n}d\sigma = \int_{E_i^*}u\cdot \hat{n}d\sigma$ (from Lemma \ref{Lemma}) for each $i$ and $E_i \cap E_j = \emptyset$ implies that $ E_i^* \cap E_j^* = \emptyset$.  Since $\bigcup_{i=1}^k E_i^* \subset \partial B_r$ and $m(A_r \setminus \bigcup_{i=1}^k E_i)\leq m(A_r \setminus K)<\epsilon$ it follows that
\begin{equation*}
\int_{\partial B_r} |u\cdot \hat{n}d\sigma| \geq \left|\int_{A_r}u\cdot \hat{n}d\sigma\right| - \epsilon \|u\|_{L^{\infty}(\partial B_{\alpha})}
\end{equation*}
Since $u \in C^{1}(D\setminus \{0\},\R^3)$ by assumption then $\|u\|_{L^\infty(\partial B_{\alpha})}<\infty$. Moreover, since $\epsilon>0$ is arbitrary we have
\begin{equation*}
\int_{\partial B_r} |u\cdot \hat{n}d\sigma| \geq \left|\int_{A_r}u\cdot\hat{n}d\sigma \right| = \left|\int_{A_r^{\alpha}}u\cdot\hat{n}d\sigma \right|
\end{equation*}
as claimed.

\begin{remark}
\begin{itemize}
\item Note that condition $|\int_{A_r^{\alpha}} u\cdot\hat{n}d\sigma| \geq Cr^{1/2}$ in the theorem implicitly requires that the Lebesgue measure of the set $A^\alpha_r$ is non zero for some $\alpha>0$ and any $0<r<\alpha$. The example of a rotating vector field $u(x)=\frac{(x_2,-x_1,0)}{|x|^\gamma}$ shows that for any $\alpha>0$, and for any $r<\alpha$ the set $A^\alpha_r$ is empty. Moreover, this example shows that the vector field $u$ can be in $L^2$ as well as not in $L^2$ depending whether or not $\gamma<4$ or $\gamma>4$.
\item 
We can easily generalize the main theorem (Theorem \ref{MainThm}) to
$L^p$ spaces ($1\leq p\leq \infty$). In fact, we just use H\"older
inequality instead of Jensen's inequality which is used in
\eqref{Jensen1} and \eqref{Jensen2}. More precisely we have the
following statement:

 If for some $\alpha>0$ and for some $C>0$
independent of $r$, $|\int_{A_r^{\alpha}} u\cdot\hat{n}d\sigma| \geq
Cr^{2(1-1/p)}$ as $r \rightarrow 0$, then $u \notin
L^p(\mathbb{R}^3)$.
\end{itemize}
\end{remark}

{\bf Acknowledgments.}
The third author  thanks  the Pacific Institute for the Mathematical Sciences
 for support of his presence there during the academic
year 2010/2011.
This paper developed during a stay of the third author as a PostDoc at the Department of Mathematics and Statistics,
University of Victoria.


\begin{thebibliography}{9}

\bibitem{Beale} J. T. Beale, T. kato, and A. Majda. Remarks on the breakdown of smooth solutions for the 3-D Euler equations. \emph{Comm. Math. Phys.}, 94 (1984) 61-66.

\bibitem{Veiga} H. Beirao da Veiga. A new regularity class for the Navier-Stokes equations in $\mathbb{R}^{n}$. \emph{Chinese Ann. Math. Ser. B}, 16 (1995) 407-412.
 A Chinese summary appears in Chinese Ann. Math. Ser. A 16 (1995), 797.


\bibitem{CKN}L. Caffarelli, R. Kohn, and L. Nirenberg. Partial regularity of suitable weak solutions of
the Navier-Stokes equations. \emph{Comm. Pure Appl. Math.}, 35
(1982) 771-831.

\bibitem{Chan}
C. H. Chan, Smoothness criteria for Navier-Stokes equations in terms
of regularity along the stream lines. \emph{Methods Appl. Anal.}, 17
(2010) 81-103.


\bibitem{CY}
C. H. Chan and T. Yoneda,
On possible isolated blow-up phenomena and regularity criterion of the 3D Navier-Stokes equation along the streamlines, submitted.

\bibitem{CSTY}
C.C. Chen,R. M.  Strain,T. P. Tsai, H. T. Yau,
 Lower bounds on the blow-up rate of the axisymmetric Navier-Stokes equations. II.
 \emph{Comm. Partial Differential Equations}, 34 (2009) 203-232.

\bibitem{CM}
A. J. Chorin and J. E. Marsden,
A mathematical introduction to fluid mechanics, Thied edition.
\emph{Springer-Verlag}, New York, 1993.









\bibitem{Esca}L. Escauriaza, G. Seregin, and V. Sverak. $L_{3, \infty}$-solutions of the Navier-Stokes
equations and backward uniqueness. \emph{Russian Math. Surveys}., 58
(2003) 211-250.




\bibitem{Foias} C. Foias and R. Temam. Some analytic and geometric properties of the solutions of the evolution
Navier-Stokes equations. \emph{J. Math. Pures Appl.(9)}, 58 (1979)
339-368.






\bibitem{Hopf}E. Hopf. Uber die Anfangswertaufgabe fur die hydrodynamischen Grundgleichungen.
\emph{Math. Nachr.}, 4 (1951) 213-231.

















\bibitem{KNSS} G. Koch, N. Nadirashvili, G. A. Seregin,  A. V. Sverak,
 Liouville theorems for the Navier-Stokes equations and applications. \emph{Acta Math.}, 203 (2009) 83-105.







\bibitem{La}
O.A. Ladyzhenskaya. Uniqueness and smoothness of generalized
solutions of Navier-Stokes equations. \emph{Zap. Naucn. Sem.
Leningrad. Otdel. Mat. Inst. Steklov.},  5 (1967) 169-185.



\bibitem{Leray}J. Leray. Sur le mouvement d'un liquide visqueux emplissant l'espace. \emph{Acta.
Math.}, 63 (1934) 183-248.






\bibitem{Prodi} G. prodi. Un teorema di unicita per le equazioni di Navier-Stokes. \emph{Ann. Mat. Pura Appl. (4)}, 48 (1959) 173-182.






\bibitem{Scheffer} V. Scheffer. Hausdorff measure and the Navier-Stokes equations. \emph{Comm. Math.
Phys.}, 55 (1977) 97-112.











\bibitem{Serrin}J. Serrin. The initial value problem for the Navier-Stokes equations. In
\emph{Nonlinear Problems Proc. Sympos., Madison, Wis.}, pages 69-98. Univ. of Wisconsin Press, Madison,
Wis., 1963.


\bibitem{Struwe} M. Struwe. On partial regularity results for the Navier-Stokes equations.
\emph{Comm. Pure Appl. Math.}, 41 (1988) 437-458.



\bibitem{Vasseur2} A. Vasseur. Regularity criterion for 3D Navier-Stokes equations in terms of the direction of the velocity.  \emph{Appl. Math.},  54  (2009)  47-52.









\end{thebibliography}
\end{document}